\documentclass[11pt]{amsart}
\usepackage{amsmath}
\usepackage{amssymb}
\usepackage{amsfonts}
\usepackage{amsthm}
\usepackage{mathabx}
\usepackage{mathrsfs}

\newcommand*\norm[1]{ \left\| #1 \right\| }

\newcommand*\Z{ \mathbb{Z} }

\newcommand*\R{ \mathbb{R} }
\newcommand*\C{ \mathbb{C} }
\newcommand*\N{ \mathbb{N} }

\newcommand*\T{\mathbb T }
\newcommand\set[1]{\left\{ #1 \right\}}

\newcommand\INNERPROD[2]{\left\langle #1, #2 \right\rangle}

\theoremstyle{theorem}

\newtheorem{Remark}{Remark}

\theoremstyle{theorem}

\newtheorem{mythm}{Theorem}[section]

\newtheorem{mylemma}{Lemma}[section]

\title[Positivity of the Lyapunov exponent]{Positivity of the Lyapunov exponent for analytic quasiperiodic operators with arbitrary finite valued background}

\date{\today}

\begin{document}

\author{Matthew Powell}
\address{Department of Mathematics, University of California, 
Irvine, CA 92697}
\email{mtpowell@uci.edu}

\maketitle

\begin{abstract}
We study lower bounds on the Lyapunov exponent associated with one-frequency quasiperiodic Schr\"odinger operators with an added finite valued background potential. We prove that, for sufficiently large coupling constant, the Lyapunov exponent is positive with a uniform (in energy and background) minoration.
\end{abstract}

\section{Introduction}
We are interested in operators on $\ell^2(\Z)$ of the form $\Delta + V,$ where $\Delta$ is the discrete Laplacian, $V(n) = \lambda v(x + n\omega) + v_1(n),$ $x,\omega \in \T,$  $v$ is a function on $\T,$ and $v_1: \Z \to \R$ is a background sequence of real numbers. 

Recently, many authors have turned their attention to the properties of Schr\"odinger operators with mixed-type potentials. One of the commonly studied models is the mixed quasiperiodic-random potential (see \cite{CaiDuarteKlein1, CaiDuarteKlein2} and references therein for known results). Of particular relevance to this note is that Cai, Duarte, and Klein recently proved a criterion for positivity of the (maximal) Lyapunov exponent, denoted $L(E)$ (see \eqref{LYAPDEF} for the definition), for mixed multifrequency quasiperiodic-random potentials, where the quasiperiodic potential is continuous \cite{CaiDuarteKlein2}. 

It is also possible, however, to consider properties of operators with quasiperiodic plus a deterministic background, such as a periodic sequence. Recently, Damanik, Fillman, and Gohlke \cite{DamanikFillmanGohlke} studied, among other more general objects, such operators where the (one-frequency) quasiperiodic potential is a trigonometric polynomial and the deterministic background is $q$-periodic, and they showed that, for large coupling constant, $\lambda,$ on the quasiperiodic potential, the Lyapunov exponent is positive. In particular, they showed that the Lyapunov exponent has an energy-independent lower bound of $\frac12 \ln(\lambda).$ As a consequence, they established Anderson localization for this model in the regime of large coupling constant.

Liu has also considered models with low-complexity backgrounds and established large-deviation estimates and modulus of continuity for the integrated density of states associated with these models \cite{WencaiDisc}; see also \cite{BourgainKachkovskiy}, where the low-complexity background was first incorporated in a localization-type argument.

This note is partially motivated by these recent results. We focus on one-frequency quasiperiodic operators with analytic potentials, along with a deterministic background consisting of a finite-range sequence$-$that is, a sequence which takes only finitely many values$-$and we prove that the Lyapunov exponent has an energy-independent and background-independent uniform minoration when the coupling constant is sufficiently large. As we can see from the existing results on mixed-type potentials, such an energy-independent result is not unexpected. The three main accomplishments here are, first, the background-independent nature of our lower bound; second, that our result holds for potentials which are analytic and not just trigonometric polynomials; and third, our backgrounds may be finite-valued and need not be periodic.

For one-frequency quasiperiodic operators with no background, positivity results go back to Herman \cite{Herman} for trigonometric polynomial potentials, and to Sorets and Spencer \cite{SoretsSpencer} for analytic potentials; see also \cite{BourgainBook}. Such results originally took the form $L(E) \geq \frac 12 \ln(\lambda).$ In the case of one-frequency analytic quasiperiodic operators with no background, this was improved to $L(E) \geq \ln(\lambda) - O(1)$ by Duarte and Klein \cite{SKLEINPOS} using a convexity argument for means of subharmonic functions, which bypasses the harmonic measure argument present in \cite{BourgainBook}. This lower bound is sharp, in the sense that $L(E) = \ln(\lambda)$ for the almost Mathieu operator. 

In the present paper, we find that it is, in fact, possible to obtain analogous results as \cite{SKLEINPOS} by carefully modifying the harmonic measure argument of \cite{BourgainBook} without appealing to convexity. Moreover, our approach is robust enough to apply when a finite-valued background is present. 

Analogous $\frac 12 \ln(\lambda)$ results for multifrequency quasiperiodic operators without background have been established by Bourgain \cite{BourgainContinuity}; the methods in the current note are unable to address the positivity of the Lyapunov exponent of multifrequency quasiperiodic operators with a background potential, but we plan to address this case in a sequel. 



More precisely, we consider the quasiperiodic operator $(H_{\lambda,\omega,x}^vu)(n) = u(n - 1) + u(n + 1) + \lambda v(x + n \omega) u(n),$ where $\omega,x \in \T$ and $v: \T \to \R$ is an analytic function which is not identically zero. We are interested in the behavior of $H = H_{\lambda,\omega,x}^v + v_1,$ where $v_1$ is a real-valued sequence on $\Z$ which takes only finitely many values. We consider lower limits of
\begin{equation}
L_N(E) = \frac 1 N \int_\T \ln\norm{M_N(x,E,\omega)} dx
\end{equation}
where
\begin{equation}
M_N(x,E,\omega) = \prod_{k = N}^1 \begin{pmatrix} E - \lambda v(x + k\omega) - v_1(k) & -1 \\ 1 & 0 \end{pmatrix}.
\end{equation}
We set
\begin{equation}\label{LYAPDEF}
L(E) = \liminf_{N \to \infty} L_N(E).
\end{equation} 
We call $L(E)$ the (lower) Lyapunov exponent. It is important to note that the limit need not exist in general. However, if the background is periodic, then we can actually easily see that the limit exists. Moreover, if the background potential is described by some ergodic process, such as a sub-shift on a finite alphabet, then we may replace $\liminf$ with $\lim$ and this definition will agree with the usual notion of (maximal) Lyapunov exponent after integration by a suitable ergodic measure associated with the background.

The recent result by Damanik, Fillman, and Gohlke, $L(E) \geq \frac 12 \ln(\lambda),$ (c.f. Theorem 4.2.5 of \cite{DamanikFillmanGohlke}) was established for potentials given by trigonometric polynomials with periodic backgrounds by appealing to Avila's global theory. The method of \cite{DamanikFillmanGohlke} does require periodicity of the background and does not easily extend to general analytic potentials. In contrast, we utilize properties of subharmonic functions to prove the sharp result $L(E) \geq \ln(\lambda) - O(1),$ which works for all analytic potentials with arbitrary finite-valued backgrounds. 


\begin{mythm}\label{MainTheorem}
Suppose $H = H^v_{\lambda,\omega,x} + v_1,$ with $H^v_{\lambda,\omega,x}$ as above. Then for any $q \in \N,$ there exists $\lambda_0 = \lambda_0(v, q),$ independent of the background, such that for any $\lambda > \lambda_0(v, q),$ and any sequence of $q$ real numbers, $v_1,$ we have $L(E) > \ln(\lambda) - O(1).$ 
\end{mythm}

\begin{Remark}
The $O(1)$ term in Theorem \ref{MainTheorem} is independent of the background and may be written down explicitly in terms of $\lambda_0$ and properties of $v.$
\end{Remark}

The background potentials we consider include periodic sequences, Sturmian sub-shifts of finite type, and Bernoulli random backgrounds. 






\section{Proof of positivity}


\begin{mylemma}\label{LineBound}
Suppose $v$ is a bounded 1-periodic non-constant analytic function on the complex strip $|\Im(z)| < \rho, \rho > 0$ 
Then for any $0 < \delta < \rho$ there is $\epsilon > 0$ depending only on $\delta, k,$ and $v$ such that, for any $k$-tuple $(E_1,...,E_k) \in \R^k,$
$$\sup_{\frac\delta 2 \leq y \leq \delta} \min_{1 \leq j \leq k} \inf_{x \in [0,1]} \left|v(x + i y) - E_j\right| > \epsilon.$$
\end{mylemma}

\begin{proof}
Fix $\delta < \rho.$ Let $\sup_{|\Im(z)| < \rho} |v(z)| = C_v < \infty.$ Observe that, if $|E_j| > 2C_v,$ then the boundedness of $v$ implies $|v(z) - E_j| > C_v$ for any $|\Im(z)| < \rho$ and $k \in \Z.$ Thus, it just suffices to establish the claim for $|E_j| \leq 2C_v.$ 

Indeed, suppose not. Using compactness of $[-2C_v, 2C_v]^k,$ we may suppose that there is some $(E_1,...,E_k) \in [-2C_v, 2C_v]^k$ such that for any $\frac \delta 2 \leq y_0 \leq \delta,$ we have
$$\inf_{x \in [0,1]} \left|v(x + i y_0) - E_j\right| = 0$$
for some $1 \leq j \leq k.$ 


Since there are infinitely many choices of $y_0,$ but only finitely many choices of $E_j,$ we must be able to find a fixed $E_j,$ a sequence $y_n$ in our desired interval, and a sequence $x_n \in [0,1]$ such that 
$$v(x_n + i y_n) - E_j = 0.$$
Since the left hand side is an analytic function, and since we are taking $x_n + i y_n$ in a compact subset of $\C,$ this analytic function must have an accumulation point of zeros in its domain, and thus it must be constant zero. This immediately implies $v(z)$ is constant, which is a contradiction. Thus, the claim holds. Uniformity of $\epsilon$ for any $k$-tuple follows from compactness of $[-2C_v, 2C_v].$
\end{proof}


With this lemma, we can now prove Theorem \ref{MainTheorem}.

\begin{proof}[{\bf{Proof of Theorem \ref{MainTheorem}}}]
Since $v$ is 1-periodic and real analytic, it has a bounded complex-analytic extension to the strip $|\Im(z)| < \rho.$ Say the extension is bounded by $C_v.$ Moreover, if we add any real number to $v,$ say $\alpha,$ then $v + \alpha$ still has a bounded complex-analytic extension to the same strip. Indeed, the bound has simply changed to $C_v + |\alpha|.$ 

We now consider the complex-analytic matrix-valued function 
\begin{equation}
M_N(z,E) = \prod_{k = N}^1 \begin{pmatrix} E - \lambda v(z + k\omega) - v_1(k) & -1 \\ 1 & 0 \end{pmatrix}.
\end{equation}
Observe that 
$$\norm{M_N(z,E)} \leq (C_v|\lambda| + |E| + \max|v_1| + 1)^N,$$
and thus
$$u_N(z) = \frac 1 N \ln \norm{M_N(z,E)}$$
is a subharmonic function on the strip $|\Im(z)| < \rho$ obeying
$$u_N(z) \leq \ln(C_v|\lambda| + |E| + \max|v_1| + 1).$$
Moreover, $M_N(z,E) \in SL_2(\C),$ so $\norm{M_N(z,E)} \geq 1.$ Thus
$$0 \leq u_N(z) \leq \ln(C_v|\lambda| + |E| + \max|v_1| + 1).$$
Our conclusion will now follow if we can bound $\int_0^1 u_N(x) dx$ from below independent of $N.$

Let $k_1,\dots,k_q$ denote the $q$ points in $\Z$ which yield the $q$ distinct values for $v_1.$  Let us consider any fixed $E \in \R$ and let $E_j = E - v_1(k_j)$ so that 
$$|v( x + i y) - E + v_1(k_j)| = |v(x + i y) - E_j|.$$
Observe that Lemma \ref{LineBound} is applicable in ($q$-tuple form) to the right hand side, so we can fix $0 < \delta = \frac\rho{100},$ and let $\epsilon > 0$ be the corresponding $\epsilon$ associated to this choice of $\delta$ in Lemma \ref{LineBound}. We know that for any $E \in \R,$ there is $\frac \delta 2 < y_0 < \delta$ such that for any $k \in \Z$ 
$$\inf_{x \in [0,1]} \left|v(x + i y_0) - \frac E \lambda + \frac{v_1(k)}{\lambda} \right| >  \epsilon.$$ Since $v$ is 1-periodic, we can extend this to the entire real line
\begin{equation}\label{KeyInequality}
\inf_{x \in \R} \left|v(x + i y_0) - \frac E \lambda + \frac{v_1(k)}{\lambda} \right| >  \epsilon.
\end{equation}
Define $\lambda_0 = \lambda_0(v) =  5C_v \epsilon^{-1}$ and fix $\lambda > \lambda_0.$ 




Consider the set $Z = \set{j \in \set{1,2,\dots,N}: |v_1(j) - E| \leq 5C_v \lambda}.$ Suppose $|Z| = k,$ and let $\set{j_1,\dots,j_{N - k}} = \set{1,\dots,N}\backslash Z.$
By induction on $N,$ \eqref{KeyInequality}, and our definition of $Z,$ we see that 
\begin{align}
\norm{M_N(a + iy_0,E)} &\geq \left|\INNERPROD{M_N(a + iy_0,E) \begin{pmatrix} 1 \\ 0 \end{pmatrix}}{\begin{pmatrix} 1 \\ 0 \end{pmatrix}}\right|\label{MNBound}\\
&\geq (\lambda \epsilon - 1)^k\prod_{n = 1}^{N - k}(|E - \lambda v(a + iy_0 + j_n\omega) - v_1(j_n)| - 1)
\end{align}
for any $a\in [0,1].$ Thus, we may improve our lower bound on $u_N$ for any $a\in \R$
$$u_N(a + iy_0) > \frac{k}{N}\ln(|\lambda|\epsilon - 1) + \frac{1}{N} \sum_{n = 1}^{N - k} \ln\left(|E - \lambda v(a + iy_0 + j_n\omega) - v_1(j_n)| - 1\right)> 0.$$

Now, we let $\mu_{a + iy_0}$ denote the harmonic measure associated to $a + iy_0$ in the complex strip $0 \leq iy \leq i\rho/2.$ In particular, $\mu_{a + i y_0}$ is a regular Borel measure on the two lines $y = 0$ and $iy = i\rho/2.$ The definition of the harmonic measure quickly yields
\begin{align}
\begin{split}\mu_{a + i y_0}\set{iy = i\rho/2} &= \frac{2 y_0}{\rho} 
\end{split}\\
\begin{split}\mu_{a + i y_0}\set{iy = 0} &= 1 - \frac{2 y_0}{\rho} 
\end{split}
\end{align}
Since $u_N$ is subharmonic, we have
\begin{align}
\begin{split}u_N(a + iy_0) 
&\leq \int_{iy = 0} u_N(x) d\mu_{a + iy_0}(x) + \int_{iy = i\rho/2} u_N(x + iy) d\mu_{a + i y_0}(x)\\
&= \int_{iy = 0} u_N(x + a) d\mu_{iy_0}(x) + \int_{iy = i\rho/2} u_N(x + a + iy) d\mu_{iy_0}(x).
\end{split}\end{align}
Here we used a change of variables and translation properties of the harmonic measure on a horizontal strip (see proof of Proposition 11.21 in \cite{BourgainBook}). 
Now we can integrate throughout in $a$ over the unit interval, and appeal to periodicity of $u$ in $x$ to obtain
\begin{align}
\begin{split}\int_0^1 u_N(x + iy_0) dx &< \int_0^1 u_N(x) dx \cdot \mu_{iy_0}\set{iy = 0} + \int_0^1 u_N(x + i \frac \rho2) dx \cdot \mu_{iy_0}\set{iy = i\frac \rho2}\\
&\leq (1 - 2y_0/\rho) \int_0^1 u_N(x) dx  + (2y_0/\rho)\int_0^1 u_N(x + i\rho/2) dx.
\end{split}\end{align}

Moreover, $$u_N(x + i\rho/2) \leq \frac{k}{N} \ln(2 C_v|\lambda|) + \frac{1}{N}\sum_{n = 1}^{N - k}\ln\left(|E - \lambda v(a + i\rho/2 + j_n\omega) - v_1(j_n)| + 1\right),$$ so
\begin{align*}
\int_0^1 u_N(x + iy_0) dx &\leq \left(1 - \frac {2y_0}{\rho}\right) \int_0^1 u_N(x) dx \\
 &\quad + \frac{2y_0}{\rho}\frac{k}{N} \ln(C_v|\lambda|) \\
 &\quad + \frac{2y_0}{\rho}\frac{1}{N}\sum_{n = 1}^{N - k}\int_0^1\ln\left(|E - \lambda v(x + i\rho/2 + j_n\omega) - v_1(j_n)| + 1\right) dx.
\end{align*}

We thus have:
\begin{align}
\begin{split}\left(1 - \frac {2y_0}{\rho}\right) &\int_0^1 u_N(x) dx  \\
&\geq \frac{k}{N}\ln(\lambda\epsilon - 1) - \frac{2y_0}{\rho}\frac{k}{N} \ln(C_v|\lambda|)\\
&\quad + \frac{1}{N} \sum_{n = 1}^{N - k} \int_0^1\ln\left(|E - \lambda v(x + iy_0 + j_n\omega) - v_1(j_n)| - 1\right)dx \\
&\quad - \frac{2y_0}{\rho}\frac{1}{N}\sum_{n = 1}^{N - k}\int_0^1\ln\left(|E - \lambda v(x + i\rho/2 + j_n\omega) - v_1(j_n)| + 1\right) dx \end{split}\\
\begin{split}\label{eq:almostthere}&\geq \frac{k}{N}\left(\ln(\lambda\epsilon - 1) - \frac{2y_0}{\rho}\ln(C_v|\lambda|)\right)\\
&\quad + \frac{2y_0}{\rho}\frac{\rho}{2y_0}\frac{1}{N} \sum_{n = 1}^{N - k} \int_0^1\ln\left(|E - \lambda v(x + iy_0 + j_n\omega) - v_1(j_n)| - 1\right)dx \\
&\quad - \frac{2y_0}{\rho}\frac{1}{N}\sum_{n = 1}^{N - k}\int_0^1\ln\left(|E - \lambda v(x + i\rho/2 + j_n\omega) - v_1(j_n)| + 1\right) dx \end{split}
\end{align}


Now we have 
$v(x + i\rho/2) = v(x + iy_0) + \eta(x),$
where $|\eta(x)| \leq C_v,$ and $|E - \lambda v(x + iy_0 + j_n\omega) - v_1(j_n)| \geq 4 C_v\lambda,$ so
\begin{align}
\begin{split}\label{eq:etatriangle}\ln(|E - &\lambda v(x + i\rho/2 + j_n\omega) - v_1(j_n)| + 1)  \\
&\leq \ln(2|E - \lambda v(x + iy_0 + j_n\omega) - \lambda \eta(x + j_n\omega) - v_1(j_n)| + 1) \end{split}.
\end{align}

It now follows that we can bound \eqref{eq:almostthere} from below using \eqref{eq:etatriangle}, the definitions of $\eta(x)$ and $Z,$ and triangle inequality to obtain

\begin{align}
\left(1 - \frac {2y_0}{\rho}\right) &\int_0^1 u_N(x) dx \\
\begin{split}&\geq \frac{k}{N}\left(\ln(\lambda\epsilon - 1) - \frac{2y_0}{\rho}\ln(C_v\lambda)\right)\\
&\quad + \frac{2y_0}{\rho}\frac{1}{N} \sum_{n = 1}^{N - k} \int_0^1 \ln\left((C_v \lambda)^{\rho/2y_0 - 1}\right)dx\end{split}\\
\begin{split}&\geq \frac{k}{N}\left(\ln(\lambda) + \ln(\epsilon) - \frac{2y_0}{\rho}\ln(\lambda) - \frac{2y_0}{\rho}\ln(C_v)\right)\\
&\quad + \frac{2y_0}{\rho}\left(\frac{\rho}{2y_0} - 1\right)\frac{N - k}{N} \ln(C_v \lambda)\end{split}\\
\begin{split}&= \frac{k}{N}\left(1 - \frac{2y_0}{\rho}\right)\ln(\lambda) + \frac k N\left(\ln(\epsilon) - \frac{2y_0}{\rho}\ln(C_v)\right)\\
&\quad + \frac{N - k}{N}\left(1 - \frac{2y_0}{\rho}\right) \ln(\lambda) + C_{v,\rho}\end{split}
\end{align}

Since $\delta/2 < y_0 < \delta = \rho/100,$ the term $\ln(\epsilon) - \frac{2y_0}{\rho}\ln(C_v)$ is a constant which may be bounded by something depending only on $v, \lambda_0,$ and $\rho,$ dividing by $1 - \frac{2y_0}{\rho}$ yields our result.

\end{proof}

\section*{Acknowledgement}
We would like to thank S. Jitomirskaya for presenting us with the problem that lead to this work, and also thank her and W. Liu for many useful suggestions and comments on earlier versions of the manuscript. In addition, we would like to thank S. Klein for suggesting that our methods can lead to an improvement in our original result. This research was partially supported by DMS-2000345, DMS-2052572, DMS-2052899, and Simons 681675.



\bibliographystyle{abbrv} 
\bibliography{PaperRevision.bib}

\begin{thebibliography}{10}

\bibitem{BourgainBook}
J.~Bourgain.
\newblock {\em Green's function estimates for lattice {S}chr\"odinger operators
  and applications}.
\newblock Princeton University Press, 2005.

\bibitem{BourgainContinuity}
J.~Bourgain.
\newblock Positivity and continuity of the {L}yapounov exponent for shifts on
  ${\T}^d$ with arbitrary frequency vector and real analytic potential.
\newblock {\em Journal d'Analyse Math\'ematique}, 96(313 - 355), 2005.

\bibitem{BourgainKachkovskiy}
J.~Bourgain and I.~Kachkovski.
\newblock Anderson localization for two interacting quasiperiodic particles.
\newblock {\em Geom. Funct. Anal.}, 29:3 -- 43, 2019.

\bibitem{CaiDuarteKlein1}
A.~Cai, P.~Duarte, and S.~Klein.
\newblock Mixed random-quasiperiodic cocycles.
\newblock {\em Preprint}, 09 2021.

\bibitem{CaiDuarteKlein2}
A.~Cai, P.~Duarte, and S.~Klein.
\newblock Furstenberg theory of mixed random-quasiperiodic cocycles.
\newblock {\em Preprint}, 01 2022.

\bibitem{DamanikFillmanGohlke}
D.~Damanik, J.~Fillman, and P.~Gohlke.
\newblock Spectral characteristics of schr{\"o}dinger operators generated by
  product systems.
\newblock {\em Preprint}, 03 2022.

\bibitem{SKLEINPOS}
P.~Duarte and S.~Klein.
\newblock {\em Continuity of the {L}yapunov exponents of linear cocycles}.
\newblock IMPA, 2017.

\bibitem{Herman}
M.~Herman.
\newblock Une methode pour minorer les exposants de {L}yapounov et quelques
  examples montrant le caract\`ere local d'un th\'eor\`eme d'{A}rnold et de
  {M}oser sur le tore de dimension 2.
\newblock {\em Comment. Math. Helv.}, 58:453 -- 562, 1983.

\bibitem{WencaiDisc}
W.~Liu.
\newblock Quantitative inductive estimates for {G}reen's functions of
  non-self-adjoin matrices.
\newblock {\em Analysis and PDE}, 2022.
\newblock To appear.

\bibitem{SoretsSpencer}
E.~Soretz and T.~Spencer.
\newblock Positive {L}yapounov exponents for schr\"odinger operators with
  quasi-periodic potentials.
\newblock {\em Comm. Math. Phys.}, 142(3):543 -- 566, 1991.

\end{thebibliography}

\end{document}